\title{Omnibus Sequences, Coupon Collection, and Missing Word Counts}
\author{\small Sunil Abraham \\
\small Oxford University \\
\small{\tt s.abraham@gmail.com} \and
\small Greg Brockman\\
\small Harvard University\\
\small {\tt gregory.brockman@gmail.com}\and
\small Stephanie Sapp \\
\small The Johns Hopkins University \\
\small {\tt sapp.stephanie@gmail.com}\and
\small{Anant P.~Godbole}\\
\small{East Tennessee State University}\\
\small{\tt godbolea@etsu.edu} }
\documentclass [12 pt]{article}
\usepackage{amsfonts}
\usepackage[dvips]{graphicx}
\usepackage{amsthm}
\usepackage{amsmath}
\usepackage{amssymb}
\usepackage[all]{xy}
\newtheorem{theorem}{Theorem}[section]

\newtheorem{lemma}[theorem]{Lemma}
\def\e{\mathbb E}
\def\a{\alpha}
\def\v{\mathbb V}
\def\ve{\varepsilon}
\def\vp{\varphi}
\def\p{\mathbb P}
\def\lr{\left(}
\def\rr{\right)}
\def\l{\lambda}
\begin{document}
\maketitle
\noindent Definition {\it Omnibus; adjective:
``Including many things or having a variety of purposes or uses."}  SHORT USAGE: {\it Omni}.

\section{Introduction}  The English translation of Leo Tolstoy's novel \emph{War and Peace} has the following notable property: it contains this paragraph as a subsequence.  If one were to write the letters and spaces that appear in the book as a string, then there would be a subsequence of the string that is identical to the string of letters and spaces in this paragraph.  The full property is more general than that -- \emph{War and Peace} contains as a subsequence \emph{any} possible string of up to 950 letters and spaces (the \TeX\ code for this paragraph has 866 characters).  That includes valid English text such as the first 950 characters of President Obama's Inaugural Address, as well as a string of 950 ``$q$''s.
\emph{War and Peace} is thus  a tome that is {\it 950-omnibus (or 950-omni)} over the 27 character alphabet $\{a, b, c,\ldots,z, SPACE\}$.  

Of course, such a text is not at all hard  to create by design.  Consider writing the string ``$abcd\ldots xyz\enspace\enspace$'' 950 times.  Clearly one could then find as a subsequence any possible string of length at most 950.  However, it seems difficult to arrive upon such a string purely by chance.  

In this paper, we will study properties of $k$-omni strings over an
alphabet of size $a$.  There are thus three areas a researcher might
pursue (i) literature and author comparisons, disputed authorship
etc. (comparative literature); (ii) constructions (combinatorics); and
(iii) behavior of random strings (statistics,
probability).  We will focus on
(iii), since (i) is best left to others, and (ii) is trivial: The
shortest string that contains all the $a^k$ words over an $a$ letter
alphabet is of length $ak$; simply write the alphabet $k$ times back
to back as done above and note that an string of length $\le ak-1$
necessarily contains a letter $\xi$ that is represented at most $k-1$
times, making the $k$-string $\xi\xi\ldots\xi$
impossible to obtain as a subsequence.

This paper is organized as follows. In Section 2, we place related problems -- in which occurrences of word segments must be in a substring -- in context.  Section 3 explores connections between omnisequences and the coupon collector problem; the key link between the two is given by what we term the ``waddle lemma."  Section 4 focuses on deriving conditions under which a random sequence is almost never or almost always $k$-omni.  Additionally, we compute exact probabilities for a sequence to be $k$-omni when its length is {\it at} the threshold value.  Section 5 is devoted to a review of some of the deeper properties of coupon collection.  We continue, in Section 6, by deriving a ``zero-infinity" threshold for the expected number of missing $k$-sequences, uncovering the fact that this threshold is not the same as that for the emergence of the omni property.  A more detailed analysis is then undertaken.  We end with potential applications in Section 7 and a list of open problems in Section 8.   We certainly raise more questions than we answer, and invite the reader to dig in with gusto into some of these questions.

\section{Universal Cycles and Cover Times}
Our interest in obtaining words as embedded subsequences aside, what can be said if each word must occur as a consecutive {\it string} rather than a scattered subsequence?  In this case, much is known.  The theories of 

(i) Universal cycles (Knuth \cite{kn}, Chung, Diaconis, and Graham \cite{ch}, Hurlbert \cite{hu}, and a Special 2009 Issue of {\it Discrete Mathematics} \cite {dm}); and 

(ii) Cover times, i.e. waiting times until all patterns occur (M\'ori \cite{mo}) 

\noindent are relevant to this question.  The following beautiful result of De Bruijn, the simple proof of which can be found, e.g., in West's graph theory textbook \cite{we}, is the starting point of all investigations on universal cycles, also known as U-cycles:
\begin{theorem}   
(DeBruijn) For each $a$ and $k$ there exists a cyclic sequence of length $a^k$ that contains as a substring {\it each} $k$-letter word over $[a]:=\{1,2,\ldots,a\}$ precisely once. 
\end{theorem}
The above theorem, exemplified for $k=3, a=2$ by the sequence 11100010,  asserts that a cyclic listing of $k$ letter words on an $a$ letter alphabet can be written down in the most efficient way possible for each $a$ and $k$.  However, omnibus sequences are far shorter, in the minimal case, than are U-cycles, since the former can be of linear length $ak$.

M\'ori \cite{mo} has done extensive work on waiting times until all patterns occur as a string.  Consider first the waiting time for occurrences of a single pattern.  Feller's classic text \cite{fe} examines this question, and the somewhat counterintuitive answer reveals that even 
the {\it expected} waiting time for a single pattern depends on the pattern; in the case of a binary alphabet, for example, the expected waiting time for $HHHHHH$ is $2^1+2^2+2^3+2^4+2^5+2^6=126$, but $HHTTHH$ occurs, on average, after just $2+4+64=70$ flips.  The underlying reason for this is that a pure head run of length six overlaps itself in six ways, but overlaps of $HHTTHH$ with itself can only occur in one, two, or six places.  One of M\'ori's \cite{mo} results on the cover times for all patterns is as follows:
\begin{theorem}
Define T(a) to be the waiting time until each of the $a^k$ patterns of length $k$ over $[a]$ has occurred as a run.  Normalize by setting 
\[Y(a)=\frac{T(a)}{a^k}-k\log a.\]
Then, with $q=\frac{1}{6}\frac{a-1}{a+1}$,
\[\sup_y\vert\p(Y(a)\le y)-F(y)\vert=O(a^{-kq}),\]
as $k\to\infty$, where $F(y)=\exp(-e^{-y})$ is the standard Gumbel distribution function.
\end{theorem}  Results on the slickest (and shortest) way of exhibiting all patterns and on average case behavior can thus both be seen to be beautiful and deep. Generalizations to cover times on graphs have been the subject of intense study; see, e.g. the monographs of Aldous \cite{al}, and Aldous and Fill \cite{alfi}. 

Consider another example of U-cycles and omnisequences: A universal cycle of length ${8\choose 3}$ of  3-subsets of $\{1,2,3,4,5,6,7,8\}$ exists; one example is given by
$$1356725\ 6823472\ 3578147\ 8245614\ 5712361\ 2467836\ 7134582\ 4681258,$$ where each block is obtained from the previous one by addition of 5 modulo 8, and the 3-letter ``word" 725, e.g., is interpreted as the 3-subset $\{2,5,7\}$.  The above U-cycle was constructed by Hurlbert \cite{hu} in a paper that studies $k$-subsets of an $n$-set for $k=3,4,5,6$, and shows how difficult the problem is for general $k$ and $n$.  Notorious in this area is the \$100 question of Chung et al.~\cite{ch}:

\medskip

\noindent {\bf Conjecture}  {\it U-cycles of $k$ element subsets of $[n]$ exist for all $k$ if the obvious divisibility condition $n\vert{n\choose k}$ is satisfied, and if $n\ge n_0(k)$ is large enough.}

\medskip

If we were looking for an for an {\it omnibus} listing of $k$-subsets of $[n]$, on the other hand, the solution would be trivial; an obvious shortest listing would be $12\ldots n$. Once again, however, we note that the length of the shortest omni sequence, namely $n$, is shorter than that of a possible shortest representation as a U-cycle, which would be of length ${n\choose k}$.  (To the best of our knowledge, average case behavior has not been studied in this case -- where one would be seeking to evaluate, e.g., $\e(W)$, where $W$ is the waiting time until each of the $k$-subsets occurs as an ordered string.    Average case omni behavior, on the other hand will follow from the discussion in the next section.)

\section {A Twist on Coupon Collection}

\noindent THE CLASSIC COUPON COLLECTOR PROBLEM:   Suppose that a ``coupon collector" wishes to collect one of each of $a$ toys that are found in cereal boxes.  It is well known (see, e.g., Feller\cite{fe})  that she expects to collect $\e(W)=a(1+\frac{1}{2}+\ldots\frac{1}{a}):=aH(1..a)\approx a[\log a+\gamma+o(1)]$ coupons, since the first purchase yields the first new toy; the expected waiting time until the second new toy is purchased is the mean of a geometric random variable with parameter $\frac{a-1}{a}$, which equals $\frac{a}{a-1}$; the third toy takes on average $\frac{a}{a-2}$ new purchases, and so on.  In addition
the variance of the waiting time in the coupon collector problem is given by $\v(W)=(a^2 \sum_{i=1}^{a-1}{1/i^2} -aH(1..a-1))<\frac{\pi^2}{6}a^2$.

It turns out that the omnisequence problem is inextricably linked to the coupon collector problem.  The following key lemma will come as no surprise to {\it cognoscenti}.

\begin{lemma}
{\it (The Basic ``Waddle\footnote{REU groups are close knit social/mathematical entities, and team members often develop their own vernacular.  In 2008, the first named author, for no apparent reason, decided to call 1-omni strings (or completed sets of coupons) {\it waddles}.} Lemma".)}
A sequence $S$ is $k$-omni if and only if there exists a pairwise-disjoint collection $P$ of completed sets of coupons (1-omni substrings of $S$) such that $|P| \geq k$.
\end{lemma}
\begin{proof} Sufficiency is easy to establish. Consider necessity.  Suppose there exist $m<k$ pairwise disjoint 1-omni substrings of $S$.  Let these be as ``tight" as possible, so that the last letter in any substring is the first occurrence of that letter.  Let these letters be $a_1,a_2,\ldots,a_m$ and let $a=(a_1a_2\ldots a_mc\ldots c)$, ($k-m$ $c$s), where $c$ is any letter not in the string after the letter $a_m$ in the $m$th string.  Then $a$ is not a subsequence of the string.  Contradiction.\end{proof}

At this juncture, it should be clear how to algorithmically find any given length $k$ string in a $k$-omnisequence $S$.  One can proceed greedily: read the omnisequence from left to right, and when the next desired letter is found, record its position.  The above proof shows that this algorithm will always yield the desired string precisely when $S$ is $k$-omni.

Similarly, we can design a greedy algorithm to determine the maximum $k$ for which a given string $S$ is $k$-omni.  Simply read across $S$ from left to right, recording each time a new 1-omni substring (complete coupon collection) is obtained.  The total number of such substrings will be the desired $k$.  Applying such an algorithm to one of several English translations of \emph{War and Peace}, the second-named author's computer demonstrated the novel to be 950-omni but not 951-omni.

\section{Threshold Behavior and Behavior at the Threshold}
Consider rolling a fair die with $a$ sides and recording the sequence of rolls obtained.  Using the fact that we are looking at {\it disjoint renewals} of the $k$ required waddles, the expected number of rolls $\e(W_{k,a})$ needed before the recorded sequence is $k$-omni on $[a]$ equals $aH(1..a)k$, since the mean waiting time for a single waddle is $aH(1..a)$.  By independence, moreover, $\v( W_{k,a} )=\v( \sum^k_{i=1}W_{1,a} )=\sum^k_{i=1}\v( W_{1,a} )= k(a^2 \sum_{i=1}^{a-1}{1/i^2} -aH(1..a-1))<k\frac{\pi^2}{6}a^2.$
Setting $W=W_{k,a}$ for simplicity, we note that $\v(W)=o(\e(W))^2$, not just for fixed $a$ as $k\to\infty$ and but also in general if at least one of $a,k$ tends to infinity.  This is our signal that $W$ will be tightly concentrated around its mean; Chebychev's inequality easily leads us to the following result:

\begin{theorem} Let $r>0$ be a constant, and fix $a\geq 2$, $n=rk$, where $n,k$ are both integers.  Then
$$\lim_{k \rightarrow\infty} {\mathbb P}({\rm Sequence\ is\ {\it k}-omni}) = \left\{
\begin{array}{ll}
0, & \text{ if $r<a H(1..a)$, or}\\
1, & \text{ if $r>a H(1..a)$}
\end{array}\right.$$
\end{theorem}
\begin{proof}  We provide just a proof of the second part of the result; the first is proved similarly.  Let $n=kaH(1..a)+\varphi(k)\sqrt{k}a$, where $\varphi(k)\to\infty$ is any sequence such that $\vp(k)=o(\sqrt{k})$.  In other words, $n$ is smaller than $(1+\ve)aH(1..a)\cdot k$.  We have
\begin{eqnarray*} \p({\rm not\ omni})&=&\p(W>kaH(1..a)+\vp(k)\sqrt{k}a)\\
&\le& \p(W-E(W)\ge\vp(k)\sqrt{k}a)\\
&\le& \frac{\v(W)}{\vp^2(k)ka^2}\\
&\le&\frac{\pi^2}{6\vp^2(k)}
\to0,\end{eqnarray*}
as asserted.  If $a\to\infty$ for fixed length words, the above proof may be modified by letting $n=kaH(1..a)+{a\vp(a)}$ where $\vp(a)\to\infty$ can grow at an arbitrarily slow rate as long as $\vp(a)=o(H(1..a))$. In general, though, we may take $n=(1+\ve)aH(1..a)k$ as long as at least one of $a,k$ tend to infinity.
\end{proof}
We now explore behavior at some threshold values of $n$, e.g. when $n=aH(1..a)k+O(1)$.  Let $P(n,k,a)$ denote the probability that a sequence of length $n$ on an alphabet $[a]$ is $k$-omni, and let $N(n,k,a)$ be the number of $k$-omni sequences of length $n$ on $[a]$.  In the binary case, when $2H(1..2)=3$, we have
\begin{theorem} $P(3k-1,k,2) = \frac{1}{2}$ for each $k$.  Furthermore, for constant $c$, as $k\rightarrow\infty$, $P(3k+c,k,2)\rightarrow\frac{1}{2}$.
\end{theorem}

\begin{proof}
We provide a constructive count of $N(n,k,2)$.  By Lemma 3.1, a string is $k$-omni precisely when it contains at least $k$ disjoint 1-omni substrings.  For any string $S=(s_1,s_2,\ldots, s_n)$, let $S_{i..j}$ denote the substring $(s_i,\ldots,s_j)$.  Given a $k$-omni string $S$, let $\{i_j\}_{j=0}^{m}$ be defined as follows:  $i_0 = 0$, and for $j>1$, $i_j$ is the smallest integer such that $S_{i_{j-1}+1..i_j}$ is 1-omni, and $m$ is as large as possible.  Now define the sequence $\{i'_j\}_{j=1}^{m}$ by $i'_j = i_j - i_{j-1}$; that is, each $i'_j$ gives the length of the relevant 1-omni substring of $S$.

Now suppose $i'_1+i'_2+\ldots+i'_k = t$ for some fixed $t$.  Since $i'_1,i'_2,\ldots,i'_k\geq 2$, elementary combinatorics gives that there are $\binom{t - k - 1}{k-1}$ solutions to this equation.  For each solution $(i_1, i_2, \ldots, i_k)$, there are precisely $2^{k + (n- t)}$ choices for $S$, since each 1-omni substring can be independently chosen to be of the form $11\ldots 10$ or $00\ldots 01$, and the remaining $n-t$ elements of $S$ can then be chosen arbitrarily.  Thus there are a total of
$$N(n,k,2) = \sum_{t = 2k}^{n} \binom{t - k - 1}{k-1} 2^{n+k- t}$$
possible $k$-omnisequences of length $n$, and the probability that a given sequence of length $n$ is $k$-omni is
\begin{align*}
P(n,k,2) & = \frac{N(n,k,2)}{2^n}\\
& = \sum_{t = 2k}^{n} \binom{t - k - 1}{k-1} 2^{k - t}\\
& = \frac{1}{2^k} \sum_{t = 0}^{n-2k} \binom{t + k - 1}{k-1} 2^{- t}.
\end{align*}
Since $\sum_{t = 0}^{k-1} \binom{t + k - 1}{k-1} 2^{-t} = 2^{k-1}$ (see, e.g.~Gould \cite{go}), $P(3k-1,k,2) = \frac{1}{2}$.  On the other hand, it is not hard to see that if $r$ is constant (or indeed if $r=o(\sqrt{k})$), $P(3k + r, k, 2)\rightarrow P(3k-1,k,2)=\frac{1}{2}$ as $k\rightarrow\infty$.  This proves the theorem.
\end{proof}

\section{Old and Recent Results on Coupon Collection}  Our intent in this section is to provide a quick review of some classical and recent work on coupon collection, keeping potential applications to omnisequences in mind at all times.  In Section 6, we will use the groundwork laid down in this section to make progress beyond Theorem 4.1.
\subsection{Approaches} Perhaps the easiest and most natural way to view the coupon collector problem is as an occupancy problem is which we place $n$ balls in $a$ urns so that the $a^n$ possibilities are equiprobable.  This is the {\it classical} approach detailed, e.g., in Feller \cite{fe}, and which yields, e.g., an exact expression for $p_b$, the probability that exactly $b$ of the $a$ coupons have been collected:  We first choose the $b$ coupons that need to be collected and then distribute the $n$ balls into the corresponding $b$ urns so that none is empty.  This yields
\[p_b=\frac{{a\choose b}S(n,b)}{a^n},\]
where 
\[S(n,b)=\sum_{j=0}^b(-1)^j{b\choose j}(b-j)^n\] are Stirling numbers of the second kind. We will return to this key
example in the next section when we consider the {\it number of
missing $k$-words}, i.e. words that are not found as a subsequence of
a given $n$-string. {\it Question:  Is there a natural way to model omnibus behavior as the successful culmination of a dependent urn model that starts as follows:  The first $k$ letters lead to a ball being tossed into one ($={k\choose k}$) of the $a^k$ boxes; the $(k+1)^{\rm st}$ letter enables one to place a ball in each of ${{k+1}\choose{k}}-{k\choose k}=k$ boxes, etc.}

  Consider next the {\it the waiting time} approach mentioned at the beginning of  Section 3:  The waiting time $W_{1,a}$ for the completion of a collection of $a$ coupons is expressed as the sum of $a$ geometric random variables with declining success probabilities 1, $(a-1)/a, (a-2)/a,\ldots 1/a$.  This representation enabled us to quickly discover the fact that  
\[\e(W_{1,a})=aH(1..a)=a(\log a+\gamma+o(1))\enspace(a\to\infty),\]
can be used to compute generating function, moments, etc., and
was the basis of the method employed by Erd\H os and R\'enyi \cite{er} to prove the extreme value limit theorem
\begin{equation}\p\lr\frac{W_{1,a}-a\log a}{a}\le x\rr\to\exp\{-e^{-x}\}\enspace (a\to\infty).\end{equation}
Finally, and of relevance to us, the geometric representation has been used (see \cite{ho} for references) to work out the asymptotics for  the distribution of the number of missing coupons, the waiting time for the $b^{\rm th}$ coupon, etc.  {\it Question:  Can we model the march towards ``omnibusness" as the realization of a sequence of dependent random variables, possibly as follows: The $r$th sequence entry generates $X_r$ new words, where the random variable $X_r$ is supported on $\{0,1,\ldots,{r\choose k}-{{r-1}\choose k}\}$?}  (We use a simpler geometric distribution model in Section 5.5 to generalize (1).)

The {\it Poisson embedding} approach is at the basis of the exposition in Aldous \cite{al}, who uses a heuristic to correctly ``guess" several  answers -- both to the coupon collector problem and to various generalizations such as waiting times until most coupons are collected; until each coupon is collected $A+1$ times (alas, this is not quite the same as having an $A+1$-omnisequence!); until each coupon is collected once when these are not equally likely to occur; until each subset in a class is hit; etc.  Lars Holst's important paper \cite{ho} shows how we may embed the placement of balls in urns (or, equivalently the drawing of balls from urns) into a Poisson process, so that many classical ``quota-related" occupancy problems such as the birthday problem, coupon collector problem, and occupancy count problem can be recast in terms of order statistics from a gamma distribution.  In addition, this method enables one to provide easier solutions to the problem of multiple coupon collection ($A+1$ coupons of each kind)  {\it Question:  Can Poisson embedding be of value in understanding Omnibus behavior?}

{\it Poisson Approximation} is another possibility that allows one to go beyond waiting time analyses.  When events are rare, the probability distribution of their counts is often well approximated by a Poisson distribution, the dependencies between the events notwithstanding.  Now, if $n$ is large compared to $a$, then a coupon being missing would be a rare occurrence.  The number of missing coupons, or the number of empty boxes, ought to have a Poisson distribution.  The Stein-Chen method of Poisson approximation, as painstakingly described in the monograph by Barbour, Holst, and Janson \cite{bhj}, is one way to quantify closeness to a Poisson distribution in an appropriate metric.  Accordingly, we find in Chapter 6 of \cite{bhj}, or in the paper \cite{baho} that features many examples related to occupancy, that the total variation distance -- between the distribution of the number of boxes with $m\ge2$ or more balls (birthday coincidences), or the number of empty boxes (missing coupons) and appropriately defined Poisson distributions -- is small under a set of conditions that permit large expected values.  Now we shall see in Section 6 that it is a rare occurrence for words to not be embedded in alphabet strings.  {\it Question:  Can the count of such words have a Poisson distribution?  Or are the dependencies such that a more complicated distribution is forced upon $X:=\sum_{j=1}^{a^k}I_j$, where the indicator variable $I_j$ equals one iff the $j^{\rm th}$ word cannot be found embedded in the $n$-string?}

Significant progress has been made in recent years towards a fine-resolution understanding of coupon collection, using methods from {\it Analytic Combinatorics}.  The paper by Zeilberger \cite{ze} has the provocative title ``How many singles, doubles, triples, etc. should the coupon collector expect?" that exemplifies the kinds of problems under attack from mathematicians from Paris to Philadelphia (\cite{fl}, \cite{fhl}, \cite{fo}, \cite{wi}, \cite{ze}).  It is undeniable, as we shall see in Sections 5.4, 5.5, and 6, that it is precisely results of this nature that will help one understand and establish the link between that which we know (waddle counts; coupons) and that which we seek to know (missing word counts).
\subsection{An Aside on Expected Values}  During Spring 2009, the senior author (AG), on sabbatical at Johns Hopkins University, taught a class entitled {\it Chance and Risk} to a small group of liberal arts students.  Students found embedded messages of their own creation in text that they randomly generated at {\tt www.random.org}.  The word ``omnibus" became part of their lexicon.  They felt some pride at the realization that another Hopkins student (SS) was a principal player in the creation of ``this omni research."  Coupon collection was thus an important class theme, and we derived the fact that $\e(W_{1,a})=aH(1..a)$ using nothing more than the fact (well-motivated if not proven) that if $X\sim{\rm Geo}(p)$, then $\e(X)=1/p$.  A student asked the following question, soon after the solution for $a=3$ ``gumballs" was presented: ``What if there aren't the same percentage of red ($R$), blue ($B$), and green ($G$) gumballs?"  The famous ``Yasin's gumball machine" problem was thus born, in which the aforementioned probabilities were 1/2, 1/3, and 1/6 respectively.  While attempting to prepare an answer key, AG realized that the only solution available in the literature for coupons that are present in unequal proportions {\it appeared}  to be the one originally given by Von Schelling \cite{sc}, which used inclusion-exclusion and was thus ``not suitable" for this class.  Below we offer an alternative.

Conditioning on the order in which the three colors appear (these are $RBG, RGB, BRG, BGR, GRB, GBR$ with respective probabilities $\frac{1}{2}\cdot\frac{1/3}{1/3+1/6}\cdot1$, $\frac{1}{6},\frac{1}{4},\frac{1}{12},\frac{1}{10},\frac{1}{15}$), we need to find the conditional expectation of the waiting time given the order of the first appearance of the colors.  Assume that the order is $RBG$.  The waiting time is then clearly $1+x+6$, where $x$ is the additional waiting time until the $B$ appears. It felt initially that this waiting time ought to be shorter than if one were waiting for a $G$ after an $R$.  {\it But it isn't}.  The conditional distribution computation reveals that in fact $x=2$, the same as the waiting time for either $B$ or $G$, given that $R$ appeared first.  Thus, in this example,
\begin{eqnarray*}\e(W_{1,3})&=&\frac{1}{3}\cdot[1+2+6]+\frac{1}{6}\cdot[1+2+3]+\frac{1}{4}\cdot[1+1.5+6]+\\&&{}\frac{1}{12}\cdot[1+1.5+2]+\frac{1}{10}\cdot[1+1.2+3]+\frac{1}{15}\cdot[1+1.2+2]=7.3,\end{eqnarray*}and, in general we have the following

\medskip

\noindent {\bf Alternative Expression for Expected Waddletime}  {\it Let balls be independently thrown into boxes labeled $1,2,\ldots,a$ so that any box hits box $j$ with probability $p_j$.  Then the expected value of the time $W=W_{1,p_1\ldots,p_a}$ until all boxes are nonempty satisfies
\[\e(W)=\prod_{j=1}^ap_j\sum_{\pi\in{\cal S}_a}q_{\pi(1)}\ldots q_{\pi(a-1)}\lr1+q_{\pi(1)}+\ldots+ q_{\pi(a-1)}\rr\]
where \[q_{\pi(j)}=\frac{1}{1-\sum_{i=1}^jp_{\pi(i)}}\]}
The above expression yields an expected 1-omni time of around 2250 until each of the letters $A$ through $Z$ are randomly obtained, if we generate the letters according to the frequency with which they actually appear in ``normal" English text. Also, the same basic technique can be used to derive a direct expression for the expected collection time for (say) two copies of each coupon in the non uniform case.  

In the next three subsections, we collect some key results that each flesh out some of the ideas from Sections 5.1 and 5.2.  Once again, we keep our ear close to the ground in the hope that we will hear something of potential application to the omnisequence problem.

\subsection{Variations}
There are many variations on the basic coupon collection theme.  We have mentioned unequal coupon probabilities and waiting times until each subset of coupons in a certain class is hit \cite{al}.  Adler and Ross \cite{ro2} study a problem that continues to examine the subset theme; they allow coupons to be collected in certain forms of subsets -- a cereal box might, for example, contain a packet that has pictures of six baseball catchers -- and the object of interest is the waiting time until each coupon is in at least one collected subset.   Myers and Wilf \cite {wi} study two simultaneous collectors.  What is the chance that they end their collections at the same time?  What is the chance that, after being tied for a while, one collector forges ahead, never to look back?  A series of such intriguing connections to ballot-like problems are given in \cite{wi}.  May \cite{may} considers coupon collection with quotas and unequal probabilities, e.g., the collection of the letters in the name ``Dr.~Pepper," where one must collect three $P$s, two $E$s, etc., and where the letters are not found with equal probabilities.

A further generalization mentioned earlier, and of particular interest to us, is coupon collection until $A+1$ copies of each coupon are scored, $A\ge1$.  Note that a (minimal) collection with $A+1$ copies of each coupon might decompose into anywhere between one and $A+1$ waddles. In the case that $A=1$, the problem goes by the name of the ``double dixie cup" problem and was first studied in the Monthly by Newman and Shepp \cite{sh}. As mentioned earlier, simpler proofs of several of Newman and Shepp's results were given by Holst \cite{ho}\footnote{See also, Myers and Wilf, where analytic combinatorial methods are used to rederive many of the results in \cite{sh}.}, who exhibited, with $V=V_{a,A+1}$ denoting  the waiting time until $A+1$ copies of each coupon are obtained and 
$$V^*=V^*_{a,A+1}=\frac{V}{a}-\log a-A\log\log a+\log A!,$$ 
that as $a\to\infty$,

(i) $V^*_{a,1}\ldots,V^*_{a,m}$ are asymptotically independent;

(ii) $\p(V^*_{a,A+1}\le u)\to\exp\{-e^{-u}\}$; and 

(iii) $\e(V_{a,A+1})=a\lr\log a+A\log\log a+\gamma-\log A!+o(1)\rr$.

\noindent It is (iii) that we draw special attention to, since it illustrates a fact that holds in many cover time problems:  If the first cover leads to duplication that is logarithmic in the size of the problem, then subsequent cover times are faster -- generating, and in a linear fashion, only $\log \log a$ ``duplicates."  Here are two further examples of this phenomenon.

\noindent {\it (i) Covering Designs}.  Let $1<t<k<n$ be integers.   A collection ${\cal A}_\l$ of $k$-element subsets (``blocks") of the $n$ element set $[n]:=\{1,\ldots,n\}$ is said to be a {\it $t-(n,k,\l)$ covering design} if each $t$-element set is contained in at least $\l$ blocks.  A natural question is:  What is the smallest size of ${\cal A}_\l$?  Erd\H os and Spencer \cite {ersp} showed that
\[\vert {\cal A}_1\vert\le{\frac{{{n}\choose {t}}}{{k\choose t}}}\lr1+{\log{k\choose t}}\rr,\]
while Godbole et al. \cite{vi} derived the bound
\[\vert {\cal A}_\l\vert\le{\frac{{n\choose t}}{{k\choose t}}}\lr1+{\log{k\choose t}}+(\l-1)\log\log{k\choose t}+O(1)\rr\] 
under some mild assumptions.

\noindent{\it (ii) $t$-Covering Arrays}. A $k\times n$ array with entries from the $q$-ary alphabet $\{0,1,\ldots, q-1\}$ is said to form a $t-(k,n,q, \l)$ {\it covering array} (\cite {sl}) if for each choice of $t$ columns, each of the $q^t$ ``words" of length $t$ may be found at least $\l$ times among the rows of the selected columns.  For fixed $n, t, q$ let $k=k(n,t,q,\l)$ be the smallest number of rows for which a $t-(k,n,q,\l)$ { covering array} exists.  Better bounds are known for $t=3, q=2, \l=1$, for example, but general upper bounds, proved in  \cite{sk} are the following ($A$ is a well specified constant)
\[k(n,t,q,1)\le A(t-1)\log n, \]and
\[k(n,t,q,\l)\le A\lr (t-1)\log n+(\l-1)\log\log n\rr.\]

\subsection{Coupon Counts}  The discussion at the end of Sections 5.1 and  5.3 illustrates that an important auxiliary variable would be counts of coupons of  different types.  Several questions may be asked.  Perhaps the first is how many coupons have been collected precisely once at the end of a successful coupon collecting quest.  Myers and Wilf \cite{wi} solve this problem.  Among their key results is the fact that on average $H(1..a)\approx \log a$ coupons have been collected precisely once at the end of a minimal completed coupon collection.  This implies that roughly $\log a$ coupons need to be collected a second time by the double dixie cup collector; the rest have already been collected twice as part of the first collection.  Thus, heuristically, the additional waiting time until these singletons turn into doubles is $a\cdot\log\log a$, as seen above \footnote{Notice also that the problem of how many $t$ sets or $t$-letter words have been covered precisely once in a covering design or $t$-covering array respectively would provide an extension of the work in \cite{wi}.}.  This fact reveals a key difference between the waiting time until each coupon is collected $k$ times and the waiting time until the sequence is $k$-omni, for which the expected value is $kaH(1..a)$.

A problem inverse to that in \cite{wi} was tackled by Badus et al. \cite{ba}.  The problem they considered was the following: How many copies are there of the $r^{\rm th}$ new coupon to be collected?  Extending the work in \cite{wi}, Zeilberger \cite{ze} gave a simpler proof of a closed form formula, first derived by Foata, Han, and Lass \cite{fhl}, for the generating function $\sum_{i=1}^\infty E(Y_i)t^i$, where $Y_i$ is the number of coupons that have been collected precisely $i$ times.   Another key contribution in \cite{fhl} is the computation of the multivariate generating function of $\p(Y_1=y_1,\ldots,Y_r=y_r, W_{1,a}=w)$.  By rephrasing the problem in terms of a coupon collector and his ordered infinite sequence of younger brothers to whom duplicates are passed on sequentially, Foata and Zeilbeger \cite{fo} derive results about the expected numbers of missing coupons in the collections of younger brothers, when $p$ brothers have complete collections.  A simpler proof of these results, for $p=1$, is provided in \cite{ro}.

\subsection{Limit Theorems} 
In this subsection, we veer the discussion back to omnisequences.  We have seen that the normalized waiting time for the coupon collector follows asymptotically a Gumbel distribution as the number of coupons gets large.  A generalization to unequal coupon probabilities is given by Neal \cite{ne}, using the Stein-Chen method \cite{bhj}.  A further generalization is provided by Martinez   \cite{ma}, who proves a ratio limit theorem for the waiting time until $A+1$ copies of $h$ coupons are obtained.  For equally likely coupons, there are a host of approximations for small values of $a$; these are of the normal, saddlepoint, and lognormal types, and a good summary may be found in \cite{ku}.  The point to emphasize is that the situation is complicated even for a single waddle if the coupon size is small. On the other hand, if $k$ is allowed to get large, then the following result on the waiting time for a sequence to become $k$-omni follows easily from the central limit theorem.
\begin{theorem} Let $W_{k,a}$ be the waiting time until a sequence $\{X_n\}_{n=1}^\infty$ of i.i.d. letters uniformly generated from $\{1,2,\ldots,a\}$ becomes $k$-omni.  Then
\[\p\lr\frac{W_{k,a}-kaH(1..a)}{\sqrt{k}S}\le x\rr\to\frac{1}{\sqrt{2\pi}}\int_{-\infty}^x\exp\{-u^2/2\}du\enspace(k\to\infty),\]
where $S$ denotes the standard deviation for the waiting time $W_{1,a}$ until a single coupon collection is obtained.  
\end{theorem}
The above result can be used to deduce, for example, that as $k\to\infty$ ($a$ fixed) the probability $P(kaH(1..a)+\sqrt{k}, k, a)$ that a string of length $kaH(1..a)+\sqrt{k}$ is $k$-omni satisfies
\[P(kaH(1..a)+\sqrt{k}, k, a)=\p\lr\frac{W_{k,a}-kaH(1..a)}{\sqrt{k}S}\le \frac{1}{S}\rr\to\Phi\lr\frac{1}{S}\rr,\]
where $\Phi$ is the standard normal distribution function; for $a=4$, we get $S=3.8$ and a limiting value of approximately 0.603.  Also, Theorem 5.1 reveals that the probability $P(kaH(1..a)+O(1), k, a)$ is asymptotically 0.5 for all $a$, thus extending Theorem 4.2.

The situation is different if $k$ is held fixed and we allow $a$ to tend to infinity.  By conditioning on the $a!^k$ orders in which letters could be generated so as to yield an omnibus sequence, we see that $W_{k,a}$ can be written as the sum of $ak$ independent geometric waiting times, with precisely $k$ having success probability $j/a$, $1\le j\le a$.  Thus, we recognize that 
\[\p\lr\frac{W_{k,a}-ka\log a}{a}\le x\rr\] represents the distribution function of the sum of $k$ identical copies of the normalized single waddle times $({W_{1,a}-a\log a})/{a}$.  The next result follows easily from the Erd\H os-R\'enyi result (1):
\begin{theorem}
\[\p\lr\frac{W_{k,a}-ka\log a}{a}\le x\rr\to\Psi(x)\enspace(a\to\infty),\]
where $\Psi$ is the distribution function of the sum of $k$ independent Gumbel variables. \end{theorem} Unfortunately, the representation of $\Psi$, as given by Nadarajah \cite{na} is not amenable to easy analysis.

\section{Missing Word Counts}
We now change our approach to the omnibus problem.  Instead of considering only sequences that contain all possible length $k$ strings, consider strings that do not necessarily attain them all.
Given a sequence $S$ of length $n$ on $[a]$, define the \emph{number of missing $k$-sequences} of $S$ to be the number of distinct $k$-sequences on $[a]$ that cannot be obtained as a subsequence of $S$.  Denote this quantity by $M=M_{k,a}=M_{k,a}(S)$, so that
 $M_{k,a}(S)=\displaystyle\sum_{T\in [a]^k} I_{k,a}(S,T)$ where the indicator variable $I_{k,a}(S,T)=I(T)$ equals 1 iff
the word $T$ is not a subsequence of $S$.  The following result is critical, and is in marked contrast to the situation when words have to occur as strings.

\begin{lemma}\label{missingequal}
For a sequence $S$ on $[a]$, the probability that a $k$-sequence is missing in $S$ is equal to the probability that any other $k$-sequence is missing in $S$.
\end{lemma}
\begin{proof}
Say $S$ is length $n$, and let $T=(t_1,t_2,\ldots,t_k)$ be any word.  Then $T$ is missing if and only if for some $0\le j\le k-1$ we make ``$j$-fold progress" towards the attainment of $T$, i.e. the first $j$ letters of $T$ can be found in $S$ as a subsequence, but not the first $j+1$.  Let us choose the spots where these $j$ letters are to appear for the first time in ${n\choose j}$ ways.  Label the spots as $i_1,\ldots,i_j$.  Now the letters prior to $i_1$ cannot contain the letter $t_1$, the letters in between $i_1$ and $i_2$ must be devoid of a $t_2$, etc.  It follows that 
\[\p(I_T=1)=\sum_{j=0}^{k-1}{n\choose j}\lr\frac{1}{a}\rr^j\lr\frac{a-1}{a}\rr^{n-j},\]
which is merely the cumulative binomial probability $B(n,k-1,1/a)$.
The above expression is dependent on only $n$ and $k$, but not on what sequence $T$ is. Notice that, for example, when $a=2$ and
$T=11\ldots1$, we should interpret the above equation as saying that
$T$ is missing if and only if the sequence $S$ contains at most
$(k-1)$ $1$s.
\end{proof}

\subsection{The Gap} We now calculate the asymptotics of the expected value $\e(M_{k,a}(S))$ over all length $n$ strings $S$, as $k\rightarrow\infty$ and $n/k=r$ is held constant.  By linearity of expectation,
 
\begin{equation}
\label{expected number of missing}
E(M_{k,a}(S)) = \displaystyle a^k\sum_{j=0}^{k-1}{n\choose j}\lr\frac{1}{a}\rr^j\lr\frac{a-1}{a}\rr^{n-j}.
\end{equation}
Now for $n \geq ak$, the maximum term in the sum (2) is the one corresponding to $j=k-1$.  This is easy to see by taking ratios of consecutive terms, and can be made precise by the following inequality from Barbour et al.~\cite{bhj}:
$${\rm Bi}(n,p)\{0,\ldots,m-1\}\le\frac{(n-m)p}{(n-1)p-(m-1)}{\rm Bi}(n,p)\{m-1\},\enspace m<np+(1-p),$$  where 
$${\rm Bi}(n,p)(A)=\sum_{j\in A}{n\choose j}p^j(1-p)^{n-j}.$$ This leads to 
\[{a^k}{\rm Bi}(n,\frac{1}{a})\{k-1\}\le E(M_{k,a}(S))\le \frac{a^k}{\lr1-\frac{ak}{n}\rr}{\rm Bi}(n,\frac{1}{a})\{k-1\}\le4a^k{\rm Bi}(n,\frac{1}{a})\{k-1\}\]if, e.g., we take $n\ge \frac{8}{9}kaH(1..a)$. Thus, $$E(M_{k,a}(S))\sim A \cdot\frac{1}{a^{n-k}}\cdot\binom{n}{k-1}(a-1)^{n-k+1}$$ for some constant $A$.  Applying Stirling's approximation with $n=rk$, we see that 
\begin{equation}E(M_{k,a}(S))\sim \frac{A(a-1)\sqrt{r}}{(r-1+o(1)){\sqrt{2\pi(r-1)k}}}\lr\frac{(a-1)^{r-1}r^r}{a^{r-1}(r-1)^{r-1}}\rr^k.\end{equation}
Now we have seen that previous asymptotic results are all couched in terms of alphabet sizes that grow to infinity.  On the other hand, omnibus behavior is best appreciated for long words from a fixed size alphabet.  Accordingly, we ask what happens to $\e(M)$ as $k\to\infty$, and find from (3) that with $$D(a,r) = \frac{ (a-1)^{r-1}r^r}{a^{r-1}(r-1)^{r-1}},$$ and $a$ fixed, $\e(M)\to0$ as $k\to\infty$ if $D(a,r)\le1$, and $\e(M)\to\infty$ ($k\to\infty$) if $D(a,r)>1$.

Notice the similarity to Theorem 4.1.  Holding the ratio $n/k=r$ constant, we again find that there is a threshold value of $r$ at which there is a sudden change in the asymptotics.  However, these threshold values are not equal to one another.  Recall, e.g., that for $k$-omni strings, the threshold ratio (prior to which the probability of a string being $k$-omni was 0, beyond which it was 1) is $2H(1..2)=3$ for $a=2$.  However, again for $a=2$, we can show that $D(2,r)=1$ when $r\approx 4.403$.  What is going on?  It appears that for values of $n$ between $3k$ and $4.403k$, sequences are omni with high probability, and yet the expected number of missing sequences is huge -- much like the evil two-valued random variable $X$ that takes on values 0 and $n^2$ with probabilities $1-1/n$ and $1/n$ respectively:  $\e(X)$ is large even though $X$ equals zero most of the time.  It appears that $M$ is {\it similarly not concentrated around its mean}.  Specifically, rare non-omni sequences tend to have unaccomplished waddles that lead to very large numbers of missing words.  We return to this question in the next section, but for now demonstrate the fact that there is a negligible ``gap" when the alphabet size is large.  In other words, as $a\to\infty$, the difference between these threshold values grows without bound, but their ratio converges to one:

\begin{theorem}
Given $a$, let $r(a)$ be the real solution to $D(a,r(a)) = 1$.  Then as $a\rightarrow\infty$, $\frac{r(a)}{a H(1..a)}\rightarrow 1$.
\end{theorem}

\begin{proof}
We show that for large $a$, $a(\log a + \log \log a) < r(a) < a(\log a + \log \log a + 2)$.  Since also $a H(1..a)\sim a \log a$, $a\to\infty$, the result will follow immediately via the squeeze theorem.
Set $r'(a) = a(\log a + \log \log a + c)$ for $a$ large and $c$ constant.  Then \begin{eqnarray*}D(a,r'(a))&=&\left(\frac{a-1}{a}\right)^{r'(a)-1}\left(\frac{r'(a)}{r'(a)-1}\right)^{r'(a)}\left(r'(a)-1\right)\\&\sim& e^{-r'(a)/a}\cdot e\cdot r'(a)(1+o(1)).\end{eqnarray*}  Thus
$$\begin{array}{lll}
D(a, r'(a)) & \sim & e^{-(\log a + \log \log a + c)}\cdot e\cdot a(\log a + \log \log a + c)\\
& = & \frac{1}{a \log a}\cdot e^{-c+1}\cdot a(\log a + \log \log a + c)\\
& \sim & e^{1-c}
\end{array}
$$
Thus if $c = 0$, $D(a,r'(a)) > 1$, and if $c = 2$, $D(a,r'(a)) < 1$.  But $D(a,r(a)) = 1$, and the result follows by monotonicity.
\end{proof}
\subsection{Understanding the Gap}  A central question is the following:  How many waddles does a random sequence of length $n$ contain?  We seek, in other words, to understand the level crossing time
\[\tau=\inf\{t: W_{1,a,1}+W_{1,a,2}+\ldots+W_{1,a,t}>n\},\]
where the $W_{1,a,j}$s are i.i.d.~random variables with distribution equal to that of a single waddle-time; if $\tau=t$ then the sequence is $t-1$-omni (there are $t-1$ waddles).  

Now if there are $r<k$ waddles, then a rather na\"ive lower bound for the number 
$M_{k,a}$ of missing words of length $k$ is $a^{k-r-1}$, as follows.  Since there are $r$ waddles, let $a_0$ be a letter not contained among the letters after the $r$th waddle is accomplished.  Furthermore, let $a_1,\ldots,a_r$ be the last letters in the $r$ successfully completed coupon collections.  Then we see that none of the words $a_1a_2\ldots a_ra_0x_1x_2\ldots x_{k-r-1}$ are contained in the string, where the $x_j$s are arbitrary.  Thus even $\sqrt{k}$ fewer waddles than expected would lead to at least $a^{\sqrt{k}-1}$ missing words.

Let $n$ be fixed.  We invoke the basic renewal equations from Section XIII.6 in Feller\cite{fe}, that state that the number $N_n$ of disjoint occurrences, among the first $n$ trials, of a recurrent event ${\cal E}$ with mean $\mu$ and variance $\sigma^2$, satisfies
\[\e(N_n)\sim\frac{n}{\mu}; \enspace \v(N_n)\sim\frac{n\sigma^2}{\mu^3}.\]  We thus see that $n$ random keystrokes on an $a$ letter keyboard are expected to contain $n/aH(1..a)$ disjoint sets of strings that do not miss any letter, and that the variance of this quantity is of order ${n\pi^2}/{6aH^3(1..a)}$.  Moreover, $N_n$ is tightly concentrated around its mean, as evidenced, e.g. by Chebychev's inequality or the Azuma-Hoeffding martingale inequality (\cite{st}) that yields, since altering one of the keystrokes $X_1,\ldots,X_n$ can change $N_n$ by at most one, 
\begin{equation}\p\lr\left\vert N_n-\frac{n}{aH(1..a)}\right\vert>\l\rr\le2\exp\{-\l^2/2n\},\end{equation}
so that for fixed $a$, the number of waddles is concentrated in an interval of width $\sqrt{n\vp(n)}$ around its expected value -- which is of order $\Theta(n)$ -- where $\vp(n)$ may tend to infinity arbitrarily slowly.  How then can we get {\it significantly} fewer waddles than expected?  To fix our ideas, we recall from (3) that for $a=2$ we expect $(27/16)^k$ missing words if $n=3k$, the threshold value for the sequence to be omni, and $\e(M_{k,a})=(256/216)^k$ if $n=4k$.  These values are derived from the linearity of expectation, and provide little insight into what {\it causes} words to be missing, the correlations between the presence or absence of words, etc.  Now, setting $a=2$ and  $n=4k$ in (4), we see that for $k$ large enough,
\begin{equation}\p(N_n<k/2)\le\p(\vert N_n-1.33k\vert\ge0.83k)\preceq(0.916)^k.\end{equation} 
Now the actual probability of having a shortfall of $0.83k$ or more waddles is certainly smaller than that given by (5), but such a shortfall would, as discussed above, lead to $2^{{k/2}}$ missing words -- and, making believe that (5) is sharp, an expected value of at least $(\sqrt{2}\cdot0.916))^k\approx(1.3)^k$ for the number of missing words.  Now, we know this is false (the correct expected value is $(256/216)^k=(1.18)^k$) but we believe the above crude analysis {\it does} add value.

To give a more specific example, we compute the probability that a sequence of length $kaH(1..a)$ has fewer than $k-\sqrt{k}$ waddles.  By Theorem 5.1, this converges to some constant $B$, and leads to the conclusion that $\e(M_{k,a})\ge B\cdot a^{\sqrt{k}}$ which certainly tends to infinity.

Fleshing out the relationship between unaccomplished coupon collections and missing word counts clearly remains a key problem that warrants deeper further investigation.

\section{Applications}
We believe that omnisequences have a large number of potential applications.  Below are some of our thoughts on the matter.

\noindent\textbf{Cryptography:} Omnisequences could provide a potential method for cryptography.  For example, suppose that Alice and Bob meet and exchange one-time pads of randomly generated letters (or even an innocuous looking copy of {\it War and Peace}).  The encryption process for a message then becomes to greedily find the position of the desired letters within the pad.  For example, given a pad of ``abfpodod$\ldots$,'' the ciphertext of ``food'' would be ``3,5,7,8.''  The decryption process simply involves reading across the pad and recording the letters that appear in the relevant positions.  Notice that both the encryption and decryption process are exceedingly simple and require very little computational resources; more complicated schema can certainly be employed.  Our results show that if we want a random pad to be able to encrypt any message of length $k$, it should have length of at least $26 H(1..26)k \approx 100k$.  (Of course, a disadvantage of this cryptographic scheme is that only about $1\%$ of the letters in the pad will actually be used.)  This is essentially a variation of (or perhaps identical to) schemes that have actually been employed in the past.

\noindent\textbf{Randomness tests:} The results of the Coupon Collector problem have been used to analyze the randomness of data samples, such as in Kendall and Babington Smith~\cite{ke}.  It is conceivable that the related but distinct results we have derived for $k$-omnisequences could be applied to randomness tests.

\noindent\textbf{Derivation of identities:} Omnisequences are a combinatorial structure that provide for multiple ways of counting any one event.  In the process of doing this research, the authors stumbled upon a number of combinatorial identities, some perhaps not noticed before.  For example, in Lemma 6.1 it was shown that $$\displaystyle\sum_{i=k}^n \binom{n}{i}(a-1)^{n-i},$$ is equal to the total number of $n$-sequences not missing a word $T$, which can also be calculated as $$\displaystyle\sum_{i=k}^n \binom{i-1}{k-1}a^{n-i}(a-1)^{i-k}$$ as follows:  Let the $i$th element of $S$, $S_i$, be the first appearance in $S$ of the last letter, $T_k$, of $T$, given that letters $T_1,\ldots,T_{k-1}$ have appeared sequentially in $S$. Now choose the positions of the relevant terms of the subsequence in $\binom{i-1}{k-1}$ ways. Consider when $T_j$, appears in $S$; each subsequent term prior to $T_{j+1}$ in $S$ has $a-1$ choices, namely not $T_{j+1}$. The $n-i$ elements after $S_i$ have $a$ choices. Hence 
$$\displaystyle\sum_{i=k}^n \binom{i-1}{k-1}a^{n-i}(a-1)^{i-k} = \sum_{i=k}^n \binom{n}{i}(a-1)^{n-i}.$$

A second such identity can be derived by considering the total number of (minimal) 1-omnisequences of length $n$ on $[a]$.  First of all, we can construct such a sequence in the following manner. Let $\alpha_1 \alpha_2 \ldots \alpha_a$ be a permutation of $a$, denoting the order in which the letters first appear; the first letter in the omnisequence is thus $\a_1$ and the last is $\a_n$. The remaining $n-a$ letters can then be placed with restriction that the letters between $\alpha_i$ and $\alpha_{i+1}$ may acquire any of the values $\alpha_1,\alpha_2,\ldots,\alpha_i$.  We note that this construction will always yield a distinct 1-omnisequence, and furthermore every 1-omnisequence can be constructed in this way.  Now if there are $l_i-1$ letters, $l_i\geq 1$, between $\alpha_i$ and $\alpha_{i+1}$, then we note that we can create $1^{l_1-1}2^{l_2-1}\cdots (a-1)^{l_{a-1}-1}$ 1-omnisequences.  Furthermore, we had $a!$ ways of creating the original permutation.  Hence we calculate the total number of 1-omnisequences of length $n$ as $a!\sum_{l_1 + \ldots + l_{a-1} = n-1} 1^{l_1-1}2^{l_2-1}\cdots (a-1)^{l_{a-1-1}} = a\sum_{l_1 + \ldots + l_{a-1} = n-1} 1^{l_1}2^{l_2}\cdots (a-1)^{l_{a-1}}$.

Alternatively, we can consider fixing the last letter of our 1-omnisequence (which can be done in $a$ ways).  Suppose that we want to have $l_1, l_2,\ldots, l_{a-1}$ copies of each of the remaining first, second, $\ldots$, $a-1$st letters in our 1-omnisequence.  Since the arrangement of these letters is arbitrary, we have that there are $\binom{n-1}{l_1, l_2, \ldots, l_{a-1}}$ sequences we can construct in this way.  Again, this construction provides a distinct 1-omnisequence, and all 1-omnisequences of length $n$ can be constructed in this manner.  Thus we obtain that there are $a\sum_{l_1 + \ldots + l_{a-1} = n-1} \binom{n-1}{l_1, l_2, \ldots, l_{a-1}}$ such omnisequences.  Combining our results, we find that
$$\sum_{l_1 + \ldots + l_{a-1} = n-1} 1^{l_1}2^{l_2}\cdots (a-1)^{l_{a-1}} = \sum_{l_1 + \ldots + l_{a-1} = n-1} \binom{n-1}{l_1, l_2, \ldots, l_{a-1}}.$$

\noindent\textbf{Linguistics:} Note that in language, letters are not randomly distributed.  Rather, they follow some weighted distribution (even this is, of course, a simplified model of language).  We note that our results can be thus be extended to languages using a weighted version of the Coupon Collector problem, such as is provided by Hermann Von Schelling~\cite{sc}.  Using the frequencies of just letters and spaces, one can calculate the expected length of a 1-omnisequence in English is about 2250.  However, our experiments with various text samples have shown that this is very rarely achieved and the real value is probably more like 4000.  In any case, this provides for some very interesting analysis and could conceivably be put to work checking, say, the degree of relationship between two languages, or testing hypotheses regarding disputed authorship (e.g.~ William Shakespeare vs.~ Francis Bacon, or Christopher Marlowe, or Edward de Vere).

\noindent\textbf{Magic and fortune-telling:} Certainly omnisequences could form the backbone of a magic trick.  In one such scheme, the magician asks an audience member to secretly compose a sentence.  He or she then theatrically shows around a piece of text that is in fact $k$-omni for large $k$.  When the audience member reveals his or her sentence, the magician can then ``magically'' find the sentence encoded in the text.  The remarkable fact is that if the text is randomly generated, it only needs to be about 100 times as long as the desired sentence for the magician to be successful every time.  Alternatively, a fortune-teller could use this technique to generate any desired message in front of a client's eyes.
\section{Open Problems}
 Questions for further investigation have been mentioned throughout the paper, but here are a few others that we consider to be central.  

\noindent(i) What is the relationship between the number of waddles in a non-omnibus sequence and the number of missing $k$-words?

\noindent  (ii) Can we approximate the distribution ${\cal L}(M_{k,a}(S))$ of missing words in an $n$ string?  

\noindent (iii) What is the variance of $M_{k,a}(S)$?, 

\noindent and, last but certainly not least, 

\noindent(iv) What are the general properties of two dimensional $n\times n$ arrays over $[a]$ that contain all $k\times k$ arrays as submatrices? (we call such arrays ``omnimosaics.")
\section{Acknowledgment}  The research of all four authors was supported by NSF-REU Grant 0552730 and conducted at ETSU during the summer of 2008.  The research of Godbole was further supported at JHU by the Acheson J. Duncan Fund for the Advancement of Research in Statistics.

\end{document}